\documentclass[letterpaper,11pt]{amsart}

\usepackage[all]{xy}                        %

\CompileMatrices                            

\UseTips                                    

\input xypic
\usepackage[bookmarks=true]{hyperref}       

\usepackage{amssymb,latexsym,amsmath,amscd}
\usepackage{xspace}

\usepackage{graphicx}

\reversemarginpar

\theoremstyle{plain}
\newtheorem{theorem}{Theorem}[section]
\newtheorem*{theorem*}{Theorem}
\newtheorem{proposition}[theorem]{Proposition}
\newtheorem{corollary}[theorem]{Corollary}
\newtheorem{lemma}[theorem]{Lemma}

\theoremstyle{definition}
\newtheorem{definition}[theorem]{Definition}

\newtheorem{remark}[theorem]{Remark}

\newcommand{\enm}[1]{\ensuremath{#1}}          %
\newcommand{\op}[1]{\operatorname{#1}}
\newcommand{\cal}[1]{\mathcal{#1}}

\newcommand{\CC}{\enm{\mathbb{C}}}

\newcommand{\ZZ}{\enm{\mathbb{Z}}}

\newcommand{\KK}{\enm{\mathbb{K}}}

\newcommand{\PP}{\enm{\mathbb{P}}}

\newcommand{\Aa}{\enm{\cal{A}}}

\newcommand{\Dd}{\enm{\cal{D}}}
\newcommand{\Ee}{\enm{\cal{E}}}
\newcommand{\Ff}{\enm{\cal{F}}}
\newcommand{\Gg}{\enm{\cal{G}}}
\newcommand{\Hh}{\enm{\cal{H}}}
\newcommand{\Ii}{\enm{\cal{I}}}

\newcommand{\Nn}{\enm{\cal{N}}}
\newcommand{\Oo}{\enm{\cal{O}}}

\renewcommand{\phi}{\varphi}
\renewcommand{\theta}{\vartheta}
\renewcommand{\epsilon}{\varepsilon}

\newcommand{\Ext}{\op{Ext}}

\renewcommand{\to}[1][]{\xrightarrow{\ #1\ }}

\newcommand{\old}[1]{}

\begin{document}

\title[A Torelli type problem for logarithmic bundles]{A Torelli type problem for logarithmic bundles over projective varieties}
\author{E. Ballico, S. Huh and F. Malaspina}
\address{Universit\`a di Trento, 38123 Povo (TN), Italy}
\email{edoardo.ballico@unitn.it}
\address{Sungkyunkwan University, 300 Cheoncheon-dong, Suwon 440-746, Korea}
\email{sukmoonh@skku.edu}
\address{Politecnico di Torino, Corso Duca degli Abruzzi 24, 10129 Torino, Italy}
\email{francesco.malaspina@polito.it}
\thanks{The second author is supported by Basic Science Research Program 2010-0009195 through NRF funded by MEST.}

\begin{abstract}
We investigate the logarithmic bundles associated to arrangements of hypersurfaces with a fixed degree in a smooth projective variety. We then specialize to the case when the variety is a quadric hypersurface and a multiprojective space to prove a Torelli type theorem in some cases.
\end{abstract}

\maketitle
\section{Introduction}
Let $\Dd=\{D_1,\cdots, D_s\}$ be an arrangement of smooth hypersurfaces on a nonsingular variety $X$. It defines the sheaf $\Omega_X^1 (\log \Dd)$ of differential $1$-forms with logarithmic poles along $\Dd$. This sheaf was originally introduced by Deligne in \cite{De} for an arrangement with normal crossings to define a mixed Hodge structure on $X\setminus \Dd$. As a special case this sheaf turns out to be locally free when $\Dd$ has simple normal crossings, and it is called a {\it logarithmic bundle}.
A natural question regarding the logarithmic bundle is to determine whether the map
$$\Phi : \Dd \mapsto \Omega_X^1 (\log \Dd)$$
is injective, i.e. can we recover $\Dd$ from its logarithmic bundle?

The answer to this question was given by Dolgachev-Kapranov \cite{DK} for hyperplane arrangements in $X=\PP^n$, the projective spaces. They proved that any arrangements of $s$ hyperplanes of $\PP^n$ determine the same logarithmic bundle if $s\leq n+2$, and that the map $\Phi$ is generically injective if $s\geq 2n+3$. Later Vall\`es filled the gap in \cite{Valles}, proving that $\Phi$ is generically injective if $s\geq n+3$. His method was to recover the hyperplane arrangement $\Dd$ as the set of unstable hyperplanes of the bundle. These results were generalized in \cite{D} and \cite{FMV} to the case of arrangements without normal crossings, dealing with the Torelli type problem.

Recently there have been several works to deal with arrangements of smooth hypersurfaces of degree at least $2$. In \cite{uy} the map $\Phi$ is injective if $\Dd=\{D\}$ consists of a single smooth divisor in $\PP^n$ and $D$ is not of Sebastiani-Thom type, i.e. the defining equation $f$ of $D$ cannot be represented as the sum
$$f(x_0, \ldots, x_n)=f_1(x_0, \ldots, x_l)+f_2(x_{l+1}, \ldots, x_n)$$
for any choice of a homogeneous coordinate $(x_0, \ldots, x_n)$ of $\PP^n$ and $0\leq l \leq n-1$. Another attempt by Angelini in \cite{Angelini} is on arrangements of hypersurfaces of degree $d$ in $\PP^n$ and she investigate the generic injectivity of $\Phi$.

In this article we ask the same question mainly over quadric hypersurfaces and multiprojective spaces. Firstly we adopt the argument of \cite{Angelini} in a general setting to prove the following (see Proposition \ref{f2}):
\begin{theorem}
For arrangements of $s$ hyperplane sections in a smooth projective variety $X\subset \PP^r$, the map $\Phi$ is generically injective if $s\geq r+3$.
\end{theorem}

Based on this result, we continue to investigate the logarithmic bundles on a $n$-dimensional smooth quadric hypersurface $Q_n$. We derived in Proposition \ref{acm} that no logarithmic bundle on $Q_n$ is arithmetically Cohen-Macaulay if $n\geq 3$ and proved in Corollary \ref{cor5} that the logarithmic bundle $\Omega_{Q_n}^1 (\log D)$ associated to a single hyperplane section $D$ is simply the pull-back of the tangent bundle on $\PP^n$ twisted by $\Oo_{\PP^n}(-2)$ under the linear projection with the apolar point to $D$ with respect to $Q_n$. It enables us to obtain the following statement (see Corollary \ref{f3} and \ref{cor5}):
\begin{theorem}
For arrangemenets of $s$ hyperplane sections in $Q_n$, the map $\Phi$ is generically injective if $s=1$ or $s\geq n+4$.
\end{theorem}

The gap with $2\leq s \leq n+3$ remains open and we get a negative answer to the case of $s=2$ in Proposition \ref{2conics}, when the arrangements are over a quadric surface $Q_2$.

In the last section we deal with the logarithmic bundles associated to a smooth hypersurface of a multiprojective space and prove the following (see Theorem \ref{u2}):
\begin{theorem}
For $X=\PP^{n_1}\times \cdots \times \PP^{n_s}$ with $s\geq 2$ and $\mathbf{a}\in \ZZ^{\oplus s}$, the map $\Phi$ is generically injective if $\Dd=\{D\}$ with $D\in |\Oo_X(\mathbf{a})|$ and $\mathbf{a} \geq (3,\cdots, 3)$, i.e. $a_i\geq 3$ for all $i$.
\end{theorem}

It is a generalization of the result in \cite{uy}, where the same result was proven in the case of $s=1$.

\section{Preliminaries}
Throughout the article, our base field $\KK$ is algebraically closed with characteristic $0$.

Let $V$ be an ($n+1$)-dimensional vector space over $\KK$ and $\PP^n=\PP V$ be the projective space parametrizing $1$-dimensional subspaces of $V$. Then a multiprojective space is defined to be $\PP^{n_1} \times \cdots \times \PP^{n_s}$ for some $(n_1, \ldots, n_s)\in (\ZZ_{\geq 1})^{\oplus s}$ where we have $\PP^{n_i}=\PP V_i$ with ($n_i+1$)-dimensional vector spaces $V_i$. We will simply denote it by $\PP^{\mathbf{n}}$ where $\mathbf{n}=(n_1, \ldots, n_s)$. Then it is embedded into $\PP U$ by the Segre embedding where $U=\otimes V_i$. Letting its projection to $i$-th factor by $f_i$, we will denote $f_1^* \Oo_{\PP^{n_1}}(a_1) \otimes \cdots \otimes f_s^* \Oo_{\PP^{n_s}}(a_s)$ by $\Oo_{\PP^{\mathbf{n}}}(a_1, \ldots, a_s)$, or simply by $\Oo_{\PP^{\mathbf{n}}} (\mathbf{a})$ where $\mathbf{a}=(a_1, \ldots, a_s)\in \ZZ^{\oplus s}$. For simplicity in notation, we denote $\Oo_{\PP^{\mathbf{n}}} (a, \ldots, a)$ by $\Oo_{\PP^{\mathbf{n}}}(a)$. For a coherent sheaf $\Ee$ on $\PP^{\mathbf{n}}$ we will denote $\Ee \otimes \Oo_{\PP^{\mathbf{n}}}(\mathbf{a})$ by $\Ee (\mathbf{a})$. We also consider an ordering on $\ZZ^{\oplus s}$ defined as follows:
$$\mathbf{a}\geq \mathbf{b} \Leftrightarrow a_i\geq b_i \text{ for all } 1\leq i\leq s.$$

Now an $n$-dimensional smooth quadric hypersurface $Q=Q_n$ is defined to be the zero set $V(f)$ in $\PP^{n+1}$ where $f$ is a quadratic polynomial in ($n+2$) variables whose partial derivatives do not vanish simultaneously. A smooth quadric surface is also a particular case of multiprojective space and so let us summarize some information about them for later use.

For a coherent sheaf $\Ee$ of rank $r$ on $Q_2$ with the Chern classes $c_1=(a,b)\in \ZZ^{\oplus 2}$ and $c_2=c\in \ZZ$, we have :
\begin{align*}
c_1(\Ee (s,t))&=(a+rs,b+rt),\\
c_2(\Ee(s,t)) &=c+(r-1)(at+bs)+2st{r \choose 2},\\
\chi (\Ee)&=(a+1)(b+1)+r-c-1
\end{align*}
for $(s,t)\in \ZZ^{\oplus 2}$.

In general, for a coherent sheaf $\Ee$ on a smooth projective variety $X$, we denote the dual sheaf of $\Ee$ by $\Ee^{\vee}$ and its cohomology group by $H^i(X, \Ee)$, or simply by $H^i(\Ee)$ if there is no confusion. We also denote its dimension as a vector space over $\KK$ by $h^i(X, \Ee)$ or $h^i (\Ee)$.

Now let us collect some definitions and well-known results about the logarithmic bundles on a smooth projective variety $X$.

\begin{definition}
An {\it arrangement} on $X$ is defined to be a set $\Dd = \{ D_1, \cdots, D_m\}$ of smooth irreducible divisors of $X$ such that $D_i \ne D_j$ for $i\ne j$. To an arrangement $\Dd$ on $X$, we can associate the {\it logarithmic sheaf} $\Omega_X^1(\log \Dd)$, the sheaf of differential $1$-forms with logarithmic poles along $\Dd$.
\end{definition}

If $\Dd$ has simple normal crossings, its logarithmic sheaf is known to be locally free and so it can be called to be the logarithmic bundle. It admits the residue exact sequence
\begin{equation}\label{seq1}
0\to \Omega_X^1 \to \Omega_X^1(\log \Dd ) \stackrel{\textrm{res}}{\to} \bigoplus {\epsilon_i}_* \Oo_{D_i} \to 0
\end{equation}
where $\epsilon_i : D_i \to X$ is the embedding and the map $\textrm{res}$ is the Poincar\'e residue morphism.

\begin{remark}
The dual of a logarithmic bundle $\Omega_X^1 (\log \Dd)$ is the sheaf of logarithmic vector fields along $\Dd$, denoted by $TX(-\log \Dd)$ (see \cite{D}). It admits the exact sequence
$$0\to TX(-\log \Dd) \to TX \to \oplus {\epsilon_i}_*\Oo_{D_i}(D_i) \to 0$$
where $TX$ is the tangent bundle of $X$.
\end{remark}

The logarithmic bundles of hyperplane arrangements on projective spaces have been investigated by many authors and below we state some results of them. Conventionally, we will denote the hyperplane arrangement on $\PP^n$ by $\Hh$.
\begin{theorem}\cite{DK}\label{DKthm}
Let $\Hh=\{H_1, \cdots, H_m\}$ be a hyperplane arrangement on $\PP^n$ with simple normal crossings. Then we have
$$\Omega_{\PP^n}^1(\log \Hh) \cong
\left\{
\begin{array}{ll}
\Oo_{\PP^n}^{\oplus (m-1)} \oplus \Oo_{\PP^n}(-1)^{\oplus (n-m+1)} &\hbox{ if $1\leq m \leq n+1$ } \\
T\PP^n(-1) &\hbox{ if $m=n+2$ }
 \end{array}
 \right.$$
\end{theorem}

In particular, the logarithmic bundle $\Omega_{\PP^n}^1 (\log \Hh)$ does not determine the arrangement $\Hh$ uniquely when $m \leq n+2$. For the other cases, i.e. $m\geq n+3$, the following result is proven.
\begin{theorem}\cite{Valles}\label{Val}
For hyperplane arrangement $\Hh=\{H_1, \cdots, H_m\}, ~m\geq n+3$ with simple normal crossings, the assignment
$$\Hh \mapsto \Omega_{\PP^n}^1 (\log \Hh)$$
is generically injective.
\end{theorem}
\begin{remark}
The result was originally proven in \cite{DK} for $m\geq 2n+3$. Recently there was a work \cite{Angelini} proving that the assignment is also generically injective when the hypersurfaces are quadrics and the number of hypersurfaces is at least ${n+2 \choose 2}+3$.
\end{remark}

\section{Tame configuration}

The main idea of Vall\`es' proof is to reconstruct the hyperplanes from the arrangement as unstable hyperplanes of the bundle $\Omega_{\PP^{n}}^1(\log \Hh)$.

\begin{definition}\cite{Valles}\cite{Angelini}
Let $\Ee$ be a vector bundle of rank $n$ on a smooth projective variety $X\subset \PP^r$. A hyperplane section $D=X\cap H$ with $H\in |\Oo_{\PP^r}(1)|$ is called an \it{unstable} hypersurface of $\Ee$ if $H^0(\Ee^\vee_{\vert_D})\ne 0$ and $H^1(\Ee^\vee_{\vert_D})\ne 0$.
\end{definition}

\begin{remark}\label{valrmk}
In \cite{Valles}, Vall\`es proved that for a generic hyperplane arrangement $\Hh=\{H_1, \cdots, H_m\}$ in $\PP^n$, the set of unstable hyperplanes of $\Omega_{\PP^n}^1(\log \Hh)$ is exactly $\Hh$ if $m\geq n+3$.
\end{remark}

\begin{lemma}\label{f1}
Let $\Dd = D_1\cup \cdots \cup D_m$ be a simple normal crossings divisor on $X$. Then $D_i$ is an unstable hypersurface of $\Omega_X^1 (\log \Dd)$ for all $i$.
\end{lemma}

\begin{proof}
Setting $D=D_i$, the exact sequence (\ref{seq1}) gives a surjection $\Omega _X^1(\log \Dd ) \to \Oo _D$ of $\Oo _X$-sheaves.
Since the tensor product is a right exact functor, we get a surjection $f: \Omega _X^1(\log \Dd )_{\vert_D} \to \Oo _D$ of $\Oo _D$-sheaves.
Since $\Dd$ has simple normal crossings, $\Omega _X^1(\log \Dd ) $ is locally free and so is $ \Omega _X^1(\log \Dd )_{\vert_D}$. Hence $f$ induces a non-zero element of $H^0(\Omega _X^1(\log \Dd )^\vee _{\vert_D})$.
\end{proof}

Let $X\subset \PP^r$ be a smooth projective variety of dimension $n\ge 2$ and let $|W| \subseteq |\Oo _X(1)|$ be the set of all hyperplane sections $X\cap H$ with $H\subset \PP^r$ a hyperplane.
Since $n\geq 2$, a standard exact sequence gives $h^0(\Oo _{X\cap H})=1$ for all $X\cap H\in |W|$ and so each element of $|W|$ is connected.

For each $D\in |W|$ there is a unique hyperplane $H\subset \PP^r$ such that $D = X\cap H$ and $D$ spans $H$. Let $\Dd := D_1\cup \cdots \cup D_m$ be a simple normal crossings divisor with $D_i\in |W|$ for all $i$. Let
$H_i\subset \PP ^r$ be the hyperplane such that $D_i = X\cap H_i$. Set $\Hh_{\Dd}:= H_1\cup \cdots \cup H_m$. Then $\Hh_{\Dd}$ is a union of $m$ distinct hyperplanes in $\PP^r$ and it is ``almost'' simple normal crossings.

\begin{definition}
An arrangement $\Dd \in |W|$ of hypersurfaces in $X$ is {\it tame} or it is a {\it tame configuration of hypersurfaces} if $\Hh_{\Dd}$ is simple normal crossings.
\end{definition}
It is clear that a general union of $m$ general elements of $|W|$ is tame. The proof of the next result is an adaptation of the proof of Theorem 5.4 in \cite{Angelini}.

\begin{proposition}\label{f2}
Let $\Dd$ be a tame configuration of $m$ hyperplane sections in $|W|$ of $X$ and let $\Hh_{\Dd}$ be its corresponding hyperplane arrangement in $\PP^r$. Assume that there is no rational normal curve $C\subset \PP^r$ such that $H_i$ is osculating $C$ for all $i$. If $m\geq r+3$, then the set of unstable hyperplane sections of $\Omega_X^1 (\log \Dd)$ in $ |W|$ is $\Dd$.
\end{proposition}

\begin{proof}
Lemma \ref{f1} asserts that each $D_i\in \Dd$ is an unstable hyperplane section. Now let us fix an unstable hypersurface $D\in |W|$ and let $H\subset \PP^r$ be its corresponding hyperplane. By Remark \ref{valrmk} it is sufficient to prove
that $H$ is an unstable hyperplane of $\Omega_{\PP^r}^1 (\log \Hh_{\Dd})$, i.e. $H^0(\Omega _{\PP^r}^1(\log \Hh_{\Dd} )^\vee _{\vert_H}) \ne 0$. Let $\Nn$ be the normal bundle of $X$ in $\PP^r$. Since each $D_i$ is smooth, each $H_i$ is transversal to $X$ and so by Proposition 2.11 in \cite{D}we have an exact sequence of vector bundles on $X$:
\begin{equation}\label{eqf1}
0 \to \Omega _X^1(\log \Dd )^\vee \to \Omega _{\PP^r}^1(\log \Hh_{\Dd} )^\vee _{\vert_X} \to \Nn \to 0.
\end{equation}
Restricting the sequence (\ref{eqf1}) to $D$, we get an exact sequence on $D$ which induces
an injective map
$$0\to H^0(D,\Omega _X^1(\log \Dd )^\vee _{\vert_D}) \to H^0(\Omega _{\PP^r}^1(\log \Hh_{\Dd} )^\vee _{\vert_D}).$$
Thus we have $H^0(\Omega _{\PP^r}^1(\log \Hh_{\Dd} )^\vee _{\vert_D}) \ne 0$.
By tensoring the following sequence with $\Omega _{\PP ^r}^1(\log \Hh_{\Dd} )^\vee$
$$0 \to \Ii _{D,H} \to \Oo _H \to \Oo _D\to 0,$$
we get that to prove that $H^0(\Omega _{\PP^r}^1(\log \Hh_{\Dd} )^\vee _{\vert_H}) \ne 0$ it is sufficient to prove that $H^1(\Ii _{D,H}\otimes \Omega _{\PP^r}^1(\log \Hh_{\Dd} )^\vee _{\vert_H})=0$, since $H^0(\Omega _{\PP^r}^1(\log \Hh_{\Dd})^\vee _{\vert_D}) \ne 0$.

Now $\Omega _{\PP^r}^1(\log \Hh_{\Dd} )$ is a Steiner bundle for $m\geq r+3$ due to Theorem 3.5 in \cite{DK}. Thus it admits the Steiner resolution:
\begin{equation}\label{sr}
0\to \Oo_{\PP^r} (-1)^{\oplus (m-r-1)} \to \Oo_{\PP^r}^{\oplus (s-1)} \to \Omega_{\PP^r}^1 (\log \Dd) \to 0.
\end{equation}
Restricting the dual of the resolution \ref{sr} to $H$ and twisting it by $\Ii _{D,H}$
we get the following exact sequence on $H$:
\begin{equation}\label{eqf3}
0 \to \Ii _{D,H}\otimes (\Omega _{\PP^r}^1(\log \Hh_{\Dd} )^\vee _{\vert_H}) \to \Ii _{D,H} ^{\oplus (s-1)} \to \Ii _{D,H}(1)^{\oplus (m-r-1)} \to 0.
\end{equation}
Since $D$ is connected and reduced, we have $h^0(\Oo _D)=1$ and so $h^1(\Ii _{D,H})=0$. We also have $H^0(\Ii _{D,H}(1))=0$ since $D$ spans $H$.
Hence the sequence (\ref{eqf3}) gives $H^1(\Ii _{D,H}\otimes (\Omega _{\PP^r}^1(\log \Hh_{\Dd} )^\vee _{\vert_H})=0$, as required.
\end{proof}

\begin{corollary}\label{k2}
Let $\Dd =D_1\cup \cdots \cup D_m \subset X$ be a general configuration of hyperplane sections and $\Dd' := D'_1\cup \cdots \cup D'_m$
be an arbitrary simple normal crossings configuration such that $\Omega _X^1(\log \Dd ) \cong \Omega _X^1(\log \Dd ')$. If $m\geq r+3$, then we have $\Dd =\Dd'$.
\end{corollary}

\begin{proof}
By Lemma \ref{f1}, each $D_i'$ is an unstable hypersurface of $\Omega _X^1(\log \Dd )$ in $|W|$ for all $i$. Since $\Dd$ is general, its corresponding arrangement $\Hh_{\Dd}$ is a general configuration of $m$ hyperplanes. Hence $\Dd$ is tame.
Since $m \ge r+3$ and $\Hh$ is general, there is no rational
normal curve $C\subset \PP^r$ such that all $H_i$'s  are osculating hyperplanes of $C$. Proposition \ref{f2} gives $D'_i\in \{D_1,\dots ,D_m\}$ for all $i$ and so we have $\Dd ' = \Dd$.
\end{proof}

\begin{lemma}\label{a2}
 Let $T$ be a reduced, connected and non-degenerate curve in $\PP^r$ with $r\geq 2$. Then we have $H^0(T,\Omega _{\PP ^r}^1(1)_{\vert _T}) =0$.
 \end{lemma}

 \begin{proof}
Assume $H^0(T,\Omega _{\PP ^r}^1(1)_{\vert_T}) \ne 0$. Let us fix homogeneous coordinates $[x_0,\dots ,x_r]$ on $T$ and look at the restriction of the Euler sequence for $T\PP^r$ to $T$:
\begin{equation}\label{eqa2}
0 \to \Omega _{\PP ^r}^1(1)_{\vert_T} \to \Oo _T^{\oplus (r+1)} \stackrel{\sigma}{\to} \Oo _T(1)\to 0,
\end{equation}
where the map $\sigma$ is induced by the map $(c_0,\dots ,c_r)\mapsto c_0x_0+\cdots +c_rx_r$. Thus any non-zero element of $H^0(T,\Omega _{\PP ^r}^1(1)_{\vert_T})$
corresponds to an $(r+1)$-uple $(c_0,\dots ,c_r) \ne (0,\dots ,0)$ of constants, since $T$ is connected. It implies that $T$ is contained in the hypersurface
$\{c_0x_0+\cdots +c_rx_r =0\}$, contradicting to the assumption.
 \end{proof}

Let $X\subseteq \PP^m$ be a smooth and non-degenerate projective variety and let $\nu _d: \PP^m \to \PP^{N_{m,d}}$ with $N_{m,d}= \binom{m+d}{m}-1$ denote the Veronese embedding of order $d$. Then the linear span of $\nu_d(X)$ in $\PP^{N_{m,d}}$ is $r$-dimensional projective $\PP^r$ with $r:= \binom{m+d}{m}-h^0(\PP ^m,\Ii _X(d)) -1$.

Let $|W_d|$ be the set of all $A\in |\Oo_X(d)|$ for the form $\nu_d(X)\cap H$ for some hyperplane $H\subset \PP^r$. We have $|W_d|=|\Oo_X(d)|$ if and only if $X$ is projectively normal in degree $d$, i.e. the restriction map $H^0(\Oo_{\PP^m}(d))\to H^0(\Oo_X(d))$ is surjective.

 \begin{proposition}\label{b1}
 For $s\geq  r+2$, let $\Dd = D_1\cup \cdots \cup D_{s}$ be a simple normal crossings divisor of $X$ such that
\begin{enumerate}
\item $D_i \in |W_d|$ for all $i \le r+2$, and
\item $\bigcup_{i=1}^{r+2}\nu_d(D_i)$ is tame in $\nu_d(X)$.
\end{enumerate}
 Let $\Dd'$ be another simple normal crossings divisor of $X$ with $\Omega_X^1 (\log \Dd)\cong \Omega_X^1 (\log \Dd ')$. If $E\in \Dd'$ is any of the divisors which is connected and not contained
in a hypersurface of degree $d$ in $\PP^m$, then we have $E\in \{D_1,\dots ,D_s\}$.
 \end{proposition}

 \begin{proof}
Write $\Aa = D_1\cup \cdots \cup D_{r+2}$ and $\Hh=H_1 \cup \cdots \cup H_{r+2}$ where $H_i$ is the hyperplane in $\PP^r$ such that $\nu_d(D_i)=H_i\cap \nu_d(X)$. Then we have an exact sequence
\begin{equation}\label{eqa3}
0 \to \Omega _X^1(\log \Aa ) \to \Omega _X^1(\log \Dd ) \to \oplus _{i=r+3}^{s} \epsilon _{i\ast} (\Oo _{D_i})\to 0.
\end{equation}
We may assume $E\ne D_i$ for all $i>r+2$ and so the restriction of the dual of (\ref{eqa3}) to $E$ gives an inclusion $j: \Omega _X^1(\log \Dd )^{\vee}_{\vert_E} \to \Omega _X^1(\log \Aa )^{\vee}_{\vert_E}$.
Since $H^0(\Omega _X^1(\log \Dd )^\vee_{\vert_ E})\ne 0$  by Lemma \ref{f1}, we also have $H^0(\Omega _X^1(\log \Aa)^\vee_{\vert_E})\ne 0$. Look at the conormal exact sequence of $\nu _d(X)$ in $\PP^r$:
 \begin{equation}\label{eqe1}
0 \to \mathcal {N}^\vee \to \Omega _{\PP ^r}^1(\log \Hh )_{\vert_{\nu _d(X)}}\to \Omega _{\nu _d(X)}^1(\log \nu _d(\Aa ))\to 0.
\end{equation}
Dualizing (\ref{eqe1}) and then restricting it to $D:= \nu _d(E)$ we get the exact sequence
\begin{equation}\label{eqe3}
0 \to \Omega _{\nu _d(X)}^1(\log \nu _d(\Aa ))^\vee _{\vert_D} \to \Omega _{\PP^r}^1(\log \Hh )^\vee _{\vert_D} \to \Nn {\vert_D} \to 0
\end{equation}
and an injective map $H^0( \Omega _{\nu _d(X)}^1(\log \nu _d(\Aa)^\vee _{\vert_D}) \to H^0(\Omega _{\PP^r}^1 (\log \Hh)^\vee_{\vert_D})$.
Thus we have $H^0(\Omega _{\PP ^r}^1(\log \Hh )^\vee_{\vert_D}) \ne 0$.

Let us fix a non-zero element $\epsilon \in H^0(\Omega _{\PP ^r}^1(\log \Hh )^\vee_{\vert_D})$. By Theorem \ref{DKthm}, we have $\Omega _{\PP ^r}^1(\log \Hh ) \cong T\PP ^r(-1)$. Since $E$ is not contained in a hypersurface of degree $d$, $D$ spans $\PP^r$. Hence there is a non-degenerate irreducible curve $T\subseteq D$ such that $\epsilon_{\vert_T} \ne 0$, which is absurd by Lemma \ref{a2}.
 \end{proof}

\section{Quadric hypersurface}

Now let us consider the logarithmic bundles over a smooth quadric hypersurface $Q=Q_n$ in $\PP^{n+1}$ with $n\geq 2$.
We recall that Kn\"{o}rrer classified all ACM bundles (i.e. bundles without intermediate cohomology)  on $Q_n$ as
direct sums of line bundles and spinor bundles (up to a twist) (see \cite{Kn}).
\begin{proposition}\label{acm}
No logarithmic bundle on $Q_n$, $n\ge 3$ is an ACM bundle.
\end{proposition}
\begin{proof}
The dual of an ACM bundle on $Q_n$ is ACM. Let us take $\Dd = D_1\cup \cdots \cup D_s$ and assume
that $TQ (-\log \Dd )$ is a direct sum of line bundles. Setting $d_i=\deg (D_i)$, we have the exact sequence
$$0 \to TQ(-\log \Dd ) \to TQ \to \oplus _{i=1}^{s}  \Oo _{D_i}(d_i) \to 0.$$
Since $TQ(-\log \Dd )$ is ACM, we get $h^0(TQ) \ge \sum _{i=1}^{s} h^0(D_i,\Oo_{D_i}(d_i))$. Note that $h^0(TQ) = (n+2)(n+1)/2$. If $D\in | \Oo _Q(d)|$ with $d\ge 2$, then we have
$$h^0(D,\Oo _D(d)) \ge h^0(D,\Oo _D(2)) \ge  \binom{n+3}{2}-2 > (n+2)(n+1)/2.$$
Thus we have $d_i =1$ for all $i$ and so $s \le (n+1)/2$. Let $H_i\subset \PP ^{n+1}$ be the hyperplane with $D_i = Q\cap H_i$. Since $s \le n$ and $D_1\cup \cdots \cup D_s$
has simple normal crossings, $\Hh := H_1\cup \cdots \cup H_s$ also has simple normal crossings. Hence $\Omega _{\PP^{n+1}}^1(\log \Hh )$ splits into a direct sum
of line bundles without $\Oo _{\PP^{n+1}}(-2)$ as its factor by Theorem \ref{DKthm}.
Now the exact sequence
$$0 \to \Oo _Q(-2)\to \Omega _{\PP^{n+1}}^1(\log \Hh )|Q \to \Omega _Q^1(\log \Dd )\to 0$$
gives a contradiction, since $\Ext^1 (\Oo_Q(a), \Oo_Q(-2))=0$ for all $a\in \ZZ$.
\end{proof}

\begin{proposition}\label{k1}
Let us fix $\Dd=D_1\cup \cdots \cup D_m$ a union of $m$ general elements $D_i \in |\Oo_Q(d)|$ with $d\geq 1$. Let $\Dd'=D_1' \cup \cdots \cup D_m'$ be an arbitrary simple normal crossings configuration with $D_i' \in |\Oo_Q(d)|$ with $\Omega_Q^1(\log \Dd) \cong \Omega_Q^1(\log \Dd')$. If $m\geq \binom{n+1+d}{n+1} -\binom{n-1+d}{n+1}+2$, then we have $\Dd=\Dd'$.
\end{proposition}

\begin{proof}
Let us define $\nu _d: \PP^{n+1} \to \PP ^N$ to be the order $d$ Veronese embedding of $\PP ^{n+1}$ with $N:= \binom{n+1+d}{n+1}-1$. Set $X= \nu _d(Q)$. The linear
span of $X$ spans an $r$-dimensional linear subspace of $\PP^N$, where $r:= \binom{n+1+d}{n+1} -\binom{n-1+d}{n+1}-1$. Since $Q$ is projectively normal, each $D\in |\Oo _Q(d)|$ is the intersection of $Q$ with a hypersurface of degree $d$ in $\PP^{n+1}$. Hence for each $D\in |\Oo _Q(d)|$ there is a unique hyperplane $H\subset \PP ^r$ such that $\nu _d(D) =X\cap H$. Now the assertion follows from Corollary \ref{k2}.
\end{proof}

Our main goal is to consider the logarithmic bundle on $Q$, specially generic injectivity of the map
$$\Phi_d^m : \Dd \mapsto \Omega_Q^1 (\log \Dd)$$
where $\Dd$ is an arrangement of $m$ hypersurfaces in $|\Oo_Q(d)|$. By Proposition \ref{k1} we obtain the following immediate consequence.

\begin{corollary}\label{f3}
The map $\Phi_d^m$ is generically injective for $m\geq \binom{n+1+d}{n+1} -\binom{n-1+d}{n+1}+2$. In particular $\Phi_1^m$ is generically injective for $m\geq n+4$.
\end{corollary}

So the question on generic injectivity of $\Phi_1^m$ remains for $1\leq m \leq n+3$. The discussion below gives the answer for the case of $m=1$.

\begin{lemma}\label{lem5}
For a point $P\in \PP^{n+1}\setminus Q$, let $\phi_P : Q \to \PP^n$ be the linear projection with the center $P$. Let us choose a smooth conic $C\subset Q$ spanning a plane $M$ not contained in $Q$

\begin{enumerate}
\item If $P \notin M$, then $\phi_P^* (T\PP^n(-1))_{\vert _C}$ has splitting type $(1,1, 0, \cdots, 0)$.
\item If $P\in M$, then $\phi _P^* (T\PP^n(-1))_{\vert_C}$ has splitting type $(2,0,\cdots, 0)$.
\end{enumerate}
In particular, if two points $P$ and $O$ in $\PP^{n+1}\setminus Q$ are distinct, then $\varphi _P^\ast (T\PP^n(-1))$ and $\varphi _O^\ast (T\PP^n(-1))$ are not isomorphic.
\end{lemma}

 \begin{proof}
First assume that $P\notin M$. Since $\varphi _P(C)$ is a smooth conic in $\PP^n$ and $T\PP^n(-1)$ is uniform along smooth conics, so $T\PP^n(-1)_{\vert _{\varphi _P(C)}}$ has splitting type $(1,1, 0, \cdots, 0)$. Since $\varphi _P$ induces an isomorphism between $C$ and $\varphi _P(C)$, $\varphi _P^\ast (T\PP^n(-1))_{\vert_{C}}$ has splitting type
$(1,1, 0, \cdots, 0)$.

Now assume $P\in M$. Since $\varphi _P(C)$ is now a line, so $T\PP^n(-1)_{\vert_{ \varphi _P(C)}}$ has splitting type $(1,0, \cdots, 0)$. Since $\varphi _P$ induces a degree two morphism between $C$ and $\varphi _P(C)$, $\varphi _P^\ast (T\PP^n(-1))_{\vert _C}$ has splitting type
$(2,0,\cdots, 0)$.

The last assertion now follows automatically.
\end{proof}

\begin{proposition}\label{v1}
Let $D$ be a smooth hyperplane section of $Q\subset \PP^{n+1}$ with $n\geq 1$. Setting $H$ to be its corresponding hyperplane in $\PP^{n+1}$, let us define two sets as follows:
\begin{align*}
E(D)&:=\{f\in \mathrm{Aut}(\PP^{n+1})~|~f(Q) = Q \text{ and } f(D) =D\}\\
S(D)&:=\{O\in \PP^{n+1}\setminus Q ~|~\forall f\in E(D), f(O)=O\}
\end{align*}
Then we have $S(D)=\{P\}$, where $P$ is the point apolar to $H$ with respect to $Q$.
\end{proposition}

\begin{proof}
Let us define $E(Q):= \{f\in \mbox{Aut}(\PP^{n+1})~|~ f(Q) = Q\}$. The restriction of $E(Q)$ to $Q$ induces a surjection onto $\mbox{Aut}(Q)$. Since $D$ spans $H$ and each element of $E(D)$ is a projective isomorphism, so we have
$$E(D) = \{f\in \mbox{Aut}(\PP^{n+1})~|~ f(Q) = Q, f(H) =H\}.$$
Since the point $P$ is apolar to $H$, we have $P\in S(D)$. Now let us assume the existence of a point $O\in S(D)$ with $O\ne P$ and let $H'\subset \PP^{n+1}$ be the hyperplane polar to $O$ with respect to $Q$. Let us set $D':= H'\cap Q$ and then $D'$ is a smooth quadric hypersurface of $Q$, since $O$ is not contained in $Q$. Now for every $f\in E(D)$ we have $f(Q)=Q$ and $f(O) =O$. So we have $f(D')=D'$ and $f(H')=H'$, which implies that $E(D') \supseteq E(D)$.

First assume $n=1$. In this case $D$ is formed by two points
$O_1,O_2$ and the lines $L_i$, $i=1, 2$, spanned  by $P$ and $O_i$ are tangent to $Q$. Note that the restriction of $E(D)$ to $Q$ induces a surjection onto $\mbox{Aut}(Q)$, $Q\cong \PP^1$ and $\mbox{Aut}(\PP ^1)$ is $3$-transitive. So there is $g\in E(D)$ such that $g(O_1) =O_1$, $g(O_2) =O_2$, implying $g\in E(D)$, but $g$ sends one of the points of $D'$
to a point of $Q\setminus (D\cup D')$. Hence $g\notin E(D')$, a contradiction.

Now assume $n\ge 2$. Let $L$ be the line spanned by $O$ and $P$. Since $P$ is not contained in $Q$, we have $L\nsubseteq Q$. Since $f(P)=P$ and $f(O) =O$,
we also have $f(L) = L$. Let us fix a hyperplane $M$ containing $P$ such that $Q\cap M$ is
a smooth quadric hypersurface of $M$. Since $P\in M$, $Q\cap M$ is smooth and $H$ is polar to $P$, $M\cap H$ is polar
to $P$ with respect to $Q\cap M$. Since $P\notin Q$, we get that $D\cap M$ is smooth. Let us use $E(D\cap M)$ and $S(D\cap M)$ for the object constructed as above with respect
to the ambient projective space. Taking homogeneous coordinates $x_0,\dots ,x_{n+1}$ with $Q = \{\sum _{i=0}^{n+1} x_i^2=0\}$ and $P= (0,0,\dots ,0,1)$, we
get that every automorphism of $Q\cap M$ sending $D\cap M$ into itself is the restriction of an automorphism of $Q$ sending $D$ into itself. Hence the restriction
map $E(D) \to E(D\cap M)$ is surjective. By induction on $n$ we get $M\cap E(D) =\{P\}$.  Since this is true for all $M$ containing $P$ and transversal to $Q$, we get that
$L$ is tangent to $Q$, i.e.  $L\cap Q$ is a unique point, say $P'$. Since $H$ is polar to $P$ and $P\in L$, we have $P'\in D$. Since $f(D) =D$ and $f(L)=L$, we get
$f(A) =A$ for all $A\in L$. Since $E(D)$ is transitive on $D$, we get that $S(D)$ is the cone with vertex $P$ and $D$ as its basis. Taking a general plane $N$ containing
$P$ and repeating the proof of the case $n=1$, we get a contradiction.
\end{proof}

\begin{corollary}\label{cor5}
Let $\Dd=\{ D\}$ be a smooth hyperplane section of $Q$ with the corresponding hyperplane arrangement $\Hh_{\Dd}=\{H\}$ in $\PP^{n+1}$. Then we have
$$\Omega_Q^1(\log \Dd) \cong \phi_P^* (T\PP^{n}(-2)),$$
where $P$ is the point apolar to $H$ with respect to $Q$. In particular, the map $\Phi_1^1$ is injective.
\end{corollary}
\begin{proof}
Note that we have $\Omega_{\PP^{n+1}}^1 (\log \Hh_{\Dd}) \cong \Oo_{\PP^{n+1}}(-1)^{\oplus (n+1)}$ and so the following exact sequence
$$0\to \Oo_Q(-2) \to \Oo_Q(-1)^{\oplus (n+1)} \to \Omega_Q^1 (\log \Dd) \to 0.$$
It implies that $\Omega_Q^1 (\log \Dd) \cong \phi_O^* (T\PP^n(-2))$ for some $O\in \PP^{n+1}\setminus Q$ and $O$ is the unique such point by Lemma \ref{lem5}. In the set-up of Proposition \ref{v1} we have
$O\in S(D)$. By Proposition \ref{v1}, the point $O$ must be the point $P$ apolar to $H$ with respect to $Q$.
\end{proof}

\begin{remark}
Note that the $\Omega_Q^1 (\log \Dd)$ associated to a smooth hyperplane section on $Q$ is stable. In the case of $n=2$, it is an element of the moduli space $\mathbf{M}_{(-1,-1)}(1)$ of stable bundles of rank $2$ on $Q$ with $c_1=\Oo_Q(-1,-1)$ and $c_2=1$. In \cite{Huh}, it is proven that $\mathbf{M}_{(-1,-1)}(1)$ is isomorphic to $\PP^3 \setminus Q$. Indeed for a stable bundle $\Ee \in \mathbf{M}_{(-1,-1)}(1)$, define $D$ to be a set of points $P\in Q$ for which there exists a section of $\Ee (1)$ whose zero is $2P$. Then $D$ turns out to be a conic on $Q$ and it gives the inverse of the isomorphism map
$$\Phi_1^1: \{ \text {the smooth conics on } Q\} \to \mathbf{M}_{(-1,-1)}(1).$$
\end{remark}


\section{Non-tame configuration}

 For any reduced effective divisor $D$ on a smooth manifold $X$
let $\widetilde{\Omega} _X^1(\log D)$ denote the logarithmic sheaf studied in \cite{D}. Let us consider the case of $X=\PP^r$ with $r\geq 3$.

\begin{definition}
$\Hh := H_1\cup \cdots \cup H_m$ be a union of $m$ distinct hyperplanes in $\PP^r$. We say that $\Hh$ has {\it normal crossings outside finitely many points} if there
is a finite set $S\subset \PP^r$ such that $\Hh _{\vert_{\PP^r \setminus S}}$ has normal crossings, or equivalently if for each $s\in \{1,\dots ,r-2\}$ and each $s$-dimensional linear
subspace $L \subset \PP^r$, at most $r-s$ of the hyperplanes $H_i$ contain $L$.
\end{definition}

 Let us denote by $\Sigma (\Hh)$ the set of all points $P\in \PP^r$ that are contained in at least $r+1$ hyperplane $H_i$. If $\Hh$ has normal crossings outside finitely many points, then we have that $\Sigma (\Hh)$ is finite and that $\Hh$ has normal crossings
if and only if $\Sigma (\Hh ) =\emptyset$ by definition. In particular $\Hh$ with $m$ hyperplanes always has normal crossings if $m\le r$.

Now let $\Hh_m$ be a hyperplane arrangement with $m$ hyperplanes with normal crossings outside finitely many points and set $\Ff _m:= \widetilde{\Omega} _{\PP^r}^1(\log \Hh _m)$. Since
$r \ge 3$ and $\Hh$ has normal crossings outside finitely many points, we have $\Ff _m:= \Omega _{\PP^r}^1(\log \Hh _m)$ by Corollary 2.8 in \cite{D} and in particular $\Ff _m$ is reflexive.
Note that $\Ff _m$ is not locally free at the points in $\Sigma (\Hh _m)$.

If $m\ge r+2$, then $\Ff _m$ fits into a Steiner's exact sequence
$$0 \to \Oo _{\PP^r}(-1)^{\oplus (m-r-1)} \to \Oo _{\PP ^r}^{\oplus (m-1)} \to \Ff _m \to 0$$
by Theorem 3.1 in \cite{D}. If $m\le r$, then $\Hh _m$ has normal crossings and so the sheaf $\Ff _m$ is locally free. Indeed we have $\Ff_m \cong \Oo _{\PP^r}^{\oplus (m-1)}\oplus \Oo _{\PP ^r}(-1)^{\oplus (r-m+1)}$ by Theorem \ref{DKthm}.

Only the case $m=r+1$ has a small query if $\Hh _{r+1}$ does not have normal crossings at some points, i.e. it is formed by $r+1$ hyperplanes through the same point $P\in \PP^r$, but has normal crossings everywhere else.  Note that any two such configurations are projectively equivalent if we do not fix the quadric hypersurface $Q$. Since $h^0(\Ff _r) =r-1$ and $h^0(\Ff _{r+2}) =r+1$,
the exact sequences
\begin{equation}\label{eqm0}
0 \to \Ff _r \to \Ff _{r+1} \to \Oo _{H_{r+1}}\to 0
\end{equation}
$$0 \to \Ff _{r+1} \to \Ff _{r+2} \to \Oo _{H_{r+2}}\to 0$$
give $h^0(\Ff _{r+1}) =r$. Hence we have a map $\beta :\Oo _{\PP^r}^{\oplus r} \to \Ff _{r+1}$, which is bijective on global section. From these two exact sequences it follows that
the sheaf $\Ff_{r+1}$ has the same Segre classes as the vector bundle $\Oo _{\PP^r}^{\oplus r}$ corresponding to the simple normal crossings case. By Proposition 3.2 in \cite{D}, $\Ff _{r+1}$
is not locally free and so $\beta$ is not an isomorphism. Since $\Ff _{r+1}$ is reflexive with $c_1(\Ff _{r+1}) =0$, so $\beta$ cannot be injective, i.e. we have $\mathrm{rank} (\mbox{Im}(\beta )) \le r-1$.
We also have $\mathrm{rank} (\mbox{Im}(\beta))\ge r-1$ since $\Ff _r \cong \Oo _{\PP^r}^{\oplus (r-1)}\oplus \Oo _{\PP^r}(-1)$ and $\Ff _r\subset \Ff _{r+1}$. Thus $\mbox{Im}(\beta)$ is a rank $(r-1)$ torsion-free subsheaf of $\Ff _{r+1}$ with $h^0(\mbox{Im}(\beta ))=r$. Hence it fits into an exact sequence
$$0 \to \Oo _{\PP^r}(-c) \to \Oo _{\PP^r}^{\oplus r} \to \mbox{Im}(\beta ) \to 0$$for some $c>0$.

Now let $r:= n+1$ and fix a smooth quadric hypersurface $Q$ in $\PP^r$. Assume also that $\Dd = \Hh _{r+1}\cap Q$ has simple normal crossings. In particular, the common point $P$ of $H_i$'s is not contained in $Q$ and so $\Ff _{r+1}$
is locally free in a neighborhood of $Q$.
We have the exact sequence
\begin{equation}\label{eqm1}
0 \to \Oo _Q(-2) \to {\Ff _{r+1}}_{\vert_Q} \to \Omega _Q^1(\log \Dd )\to 0.
\end{equation}
From (\ref{eqm1}) we get $h^0(\Omega _Q^1(\log \Dd )) = h^0({\Ff _{r+1}}_{\vert_Q})$. Note that $\Ff _{r+1} \subset \Ff _{r+2}$ and $\Ff_{r+2}$ is a Steiner sheaf and so we get
$h^0(\Ff _{r+1}(-2))=0$. It implies that the restriction map
$$\rho :H^0(\Ff _{r+1})\to H^0({\Ff _{r+1}}_{\vert_Q})$$ is injective. Since $h^0(\Ff_r(-2))=0$, the sequence (\ref{eqm0}) gives $h^1(\Ff_{r+1}(-2))=0$. It implies the surjectivity of $\rho$ and so it is an isomorphism.

Let $\Gg \subset \Ff _{r+1}$ be the image of the evaluation map $H^0({\Ff _{r+1}}_{\vert_Q}) \otimes \Oo_Q  \to {\Ff _{r+1}}_{\vert_Q}$.

\begin{proposition}\label{ut1}
For an arrangement $\Dd$ of $m$ hyperplane sections in $Q_2$ with simple normal crossings, the logarithmic bundle $\Omega_Q^1 (\log \Dd)$ is not globally generated only if $m \leq r+1$.
\end{proposition}
\begin{proof}
It suffices to consider the case of $m=r+1$.
Since the map $\beta$ has rank $r-1$ and $\rho$ is bijective, $\Gg$ has at most rank $r-1$. We have $h^0(\Gg )=r$ and the natural map $u: \Gg \to \Omega _Q^1(\log \Dd )$ induces
a bijection on global sections. Note that every surjective map from a torsion-free sheaf of rank at most $r$ onto a vector bundle of rank $r$ is an isomorphism. Thus if $u$ is not an isomorphism, then $\Omega _Q^1(\log \Dd )$ is not globally generated.
\end{proof}

\section{Smooth quadric surface}

In this section, our main goal is to investigate the logarithmic bundles on a smooth quadric surface $Q=Q_2$ and specially the generic injectivity of the mapping
$$\Phi_{(a,b)}^m : \Dd \mapsto \Omega_Q^1 (\log \Dd)$$
where $\Dd$ is an arrangement of $m$ hypersurfaces of bidegree $(a,b)$ in $Q$ with simple normal crossings.

\begin{remark}
Let us assume that $\Dd=\{D_1, \cdots, D_m\}$ is an arrangement of smooth curves $D_i$ of bidegree $(a_i, b_i)$ on $Q$ with simple normal crossings. From the sequence (\ref{seq1}), we have
\begin{align*}
c_1(\Omega_Q^1(\log \Dd))&=(-2+\alpha , -2+\beta),\\
c_2(\Omega_Q^1(\log \Dd))&=4-2\alpha-2\beta+\sum _i 2a_ib_i + \sum_{i < j} (a_ib_j+a_jb_i),
\end{align*}
where $\alpha=\sum a_i$ and $\beta=\sum b_j$.
\end{remark}

Let us start with the arrangement with simplest hypersurfaces.

\begin{proposition}\label{line}
Let $\Dd=\{A_1, \cdots, A_a, B_1, \cdots, B_b\}$ be an arrangement of $a+b$ lines on $Q$ with $A_i \in |\Oo_Q(1,0)|$ and $B_j\in |\Oo_Q(0,1)|$. Then we have
$$\Omega_Q^1 (\log \Dd) \cong \Oo_Q(-2+a,0)\oplus \Oo_Q(0,-2+b).$$
\end{proposition}
\begin{proof}
Let us first consider the case of $(a,b)=(1,0)$. Then we have the sequence
$$0\to \Oo_Q(-2,0)\oplus \Oo_Q(0,-2) \to \Omega_Q^1 (\Dd) \to \Oo_{A_1} \to 0.$$
Note that the dimension of $\Ext^1 (\Oo_{A_1}, \Oo_Q(-2,0))$ is $h^1(\Oo_{A_1}(-2))=1$ and similarly we have $\Ext^1 (\Oo_{A_1}, \Oo_Q(0,-2))=0$. Thus there exists a uniquely determined extension of $\Oo_{A_1}$ by $\Omega_Q^1$ and it must be $\Oo_Q(-1,0)\oplus \Oo_Q(0,-2)$. Now assume that the assertion is true for $(a,0)$ to use induction. For the case of $(a+1,0)$, we have the sequence
$$0\to \Oo_Q(-2+a, 0)\oplus \Oo_Q(0,-2) \to \Omega_Q^1(\log \Dd) \to \Oo_{A_{a+1}} \to 0.$$
By the same computation as above, we have the unique such extension and so $\Omega_Q^1 (\log \Dd ) \cong \Oo_Q(-2+a+1, 0)\oplus \Oo_Q(0,-2)$. So the assertion follows in the case when either $a$ or $b$ is zero.

Now let us deal with the case when $a$ and $b$ are at least $1$. The logarithmic bundle $\Omega_Q^1 (\log \Dd)$ is an extension of $(\oplus \Oo_{A_i})\oplus (\oplus \Oo_{B_j})$ by $\Oo_Q(-2,0)\oplus \Oo_Q(0,-2)$. Note that we have
$$\Ext^1 (\oplus \Oo_{A_i}, \Oo_Q(0,-2))=\Ext^1 (\oplus \Oo_{B_j}, \Oo_Q(-2,0))=0.$$
Thus $\Omega_Q^1 (\log \Dd)$ corresponds to an element $\epsilon$;
$$\epsilon \in \Ext^1 (\oplus \Oo_{A_i}, \Oo_Q(-2,0))\oplus \Ext^1 (\oplus \Oo_{B_j}, \Oo_Q(0,-2)).$$
From the first argument above, we observe that the first factor of $\epsilon$ with $\Ext^1(\oplus \Oo_{A_i}, \Oo_Q(0,-2))=0$ generates $\Oo_Q(-2+a, 0)\oplus \Oo_Q(0,-2)$ and similarly the second factor generates $\Oo_Q(-2,0)\oplus \Oo_Q(0,-2+b)$.
Thus $\epsilon$ corresponds to the bundle $\Oo_Q(-2+a,0)\oplus \Oo_Q(0,-2+b)$.
\end{proof}

In particular, we obtain that $\Omega_Q^1 (\log \Dd)$ with $\Dd\in |\Oo_Q(a,b)|$ consisting of lines on $Q$, is an ACM bundle if and only if $1\leq a,b \leq 3$. In general we obtain the following:

\begin{corollary}
Let $\Dd$ be an arrangement of smooth curves on $Q$ with simple normal crossings. Then $\Omega_Q^1 (\log \Dd)$ is not an ACM bundle, except when $\Dd\in |\Oo_Q(a,b)|$ consists of lines on $Q$ with $1\leq a,b \leq 3$.
\end{corollary}
\begin{proof}
If $\Dd:=\{D_1, \cdots, D_m\}$ consists of $m$ smooth curves, then it admits the sequence
\begin{equation}\label{seq4}
0\to \Oo_Q(-2,0)\oplus \Oo_Q(0,-2) \to \Omega_Q^1 (\log \Dd) \to \oplus_{i=1}^m \Oo_{D_i} \to 0.
\end{equation}
If $\Ee:=\Omega_Q^1(\log \Dd)$ is ACM, then equivalently we have $h^1(\Ee(t,t))=0$ for all $t\in \ZZ$. From the sequence (\ref{seq4}), we have $\sum h^1(\Oo_{D_i})=0$. It implies that each $D_i$ is a rational curve and so it is either a line in a ruling or a rational normal curve of bidegree $(c,1)$ (or $(1,c)$) with $c\geq 1$. We can exclude the latter case since we would have $h^1(\Oo_{D_i}(-1,-1))=h^0(\Oo_{\PP^1}(c-1))>0$, which is impossible from (\ref{seq4}) twisted by $\Oo_Q(-1,-1)$ and the fact that $h^2(\Oo_Q(-3,-1))=h^2(\Oo_Q(-1,-3))=0$.  Now the assertion follows from Proposition \ref{line}.
\end{proof}

\begin{proposition}\label{b2.00}
For $s\geq 5$, let $\Dd = D_1\cup \cdots \cup D_s$ be a normal crossings divisor with $\sum_{i=1}^4 D_i\in |\Oo_Q(2,2)|$. For a fixed normal crossings divisor $\Dd'$ with $\Omega _Q^1(\log \Dd )\cong \Omega _Q^1(\log \Dd ')$, let $E$ be any irreducible component of $\Dd'$ which is not a line. Then we have $E=D_i$ for some $5\leq i \leq s$.
\end{proposition}

\begin{proof}
Note that $\Dd$ cannot be a union of lines due to Proposition \ref{line} and the assumption that $E$ is not a line. So there is $C\in \{D_5,\dots ,D_s\}$, say $C\in |\Oo _Q(a,b)|$, with
$a>0$ and $b>0$. Set $\Hh := D_1\cup \cdots \cup D_4 \cup C$ and $\Ff := \Omega _Q^1(\log \Hh)$. By Proposition \ref{line} the vector bundle $\Ff$ fits into the exact sequence
\begin{equation}\label{eqbb1}
0 \to \Oo _Q^{\oplus 2} \to \Ff \to \Oo _C \to 0
\end{equation}
and in particular we have $h^0(\Ff )=3$. From (\ref{eqbb1}) we also get that $\Ff$ is globally generated outside $C$.

Assume $E\notin \{D_5,\dots ,D_s\}$. By Lemma \ref{f1} we have
$h^0(\Ff ^\vee _{\vert_E}) > 0$. Since $\Ff$ is globally generated outside $C$ and $C\cap E$ is finite, the sheaf $\Ff _{\vert E}$ is a locally free sheaf
whose global sections span it outside finitely many points. Hence $h^0(\Ff ^\vee _{\vert_E}) =0$, a contradiction.
\end{proof}

Let us recall that for a smooth curve $C\in |\Oo _Q(u,v)|$ the tangent bundle of $C$ has no non-trivial global section if and only if $C$ has genus $\ge 2$, i.e. $(u,v)\gneq(2,2)$.

\begin{proposition}\label{b5}
Let $\Dd$ and $\Dd '$ be two simple normal crossings divisors of $Q$ with $\Omega_Q^1(\log \Dd) \cong \Omega_Q^1 (\log \Dd ')$. Assume the existence of $C\in \Dd$, say $C\in |\Oo _Q(a,b)|$, and $E\in \Dd'$, say $E\in |\Oo _Q(c,d)|$ such that $(a,b)\gneq (2,2)$, $c>0$, $d>0$ and $ad+bc > 2ab$. Then we have $E\in \Dd$.
\end{proposition}

\begin{proof}
Assume $E\notin \Dd$. Look at the exact sequence
\begin{equation}\label{eqb5.0}
0 \to TQ(-\log C) \to \Oo _Q(2,0)\oplus \Oo _Q(0,2) \to \Oo _C(a,b)\to 0
\end{equation}
Since $E\cap C$ is finite, we have $\mathrm{Tor} ^1_{\Oo _Q}(\Oo _C(a,b),\Oo _E) =0$ and so we get an exact sequence on $E$:
$$0 \to TQ(-\log C)_{\vert_E} \to \Oo _E(2,0)\oplus \Oo _E(0,2) \to \Oo _{C\cap E}(a,b)\to 0$$
which induces in cohomology a map
$$f: H^0(E,\Oo _E(2,0)\oplus \Oo _E(0,2)) \to H^0(E\cap C,\Oo _{C\cap E}(a,b)).$$
As in the proof of Proposition \ref{b1} we get $H^0(E,TQ(-\log C)_{\vert_E}) \ne 0$ and so we have $\ker (f) \ne 0$. Take $u\in \ker (f)$.  Since $H^0(TQ(-\log C)) =0$, the restriction map
$\rho : H^0(Q,TQ) \to H^0(C,TQ_{\vert_C})$ is injective. Since $H^0(C,TC)=0$, the normal bundle sequence of $C\subset Q$ gives
that the map $\rho ':  H^0(C,TQ_{\vert_C}) \to H^0(C,\Oo _C(a,b))$ is injective.

On the other hand we have $H^1(Q,TQ(-\log E)) =0$ by K\"{u}nneth formula and so the restriction
map $H^0(Q,TQ) \to H^0(Q,TQ_{\vert_E})$ is surjective. Thus there is $u'\in H^0(Q,TQ)$ such that $u'_{\vert_E} = u$. Let $w \in H^0(C,\Oo _C(a,b))$ be the image of $u'$ by
the map induced by (\ref{eqb5.0}). Since the image of $u$ vanishes on the set $C\cap E$, so $w$ vanishes on $C\cap E$ with degree $ad+bc$. Now by the assumption that $ad+bc > 2ab
= \deg (\Oo _C(a,b))$, we have $w=0$. Hence $u'$ comes from a non-zero section of $TQ(-\log C)$. Since $a\ge 2$ and $b\ge 2$ we have $h^0(TQ(-\log C))=0$, a contradiction.
\end{proof}

Now let us deal with the case when $\Dd$ consists of $m$ hyperplane sections on $Q$. By Corollary \ref{f3} and Corollary \ref{cor5}, we know that the map $\Phi_{(1,1)}^m$ is generically injective for $m\geq 6$ and $m=1$. So let us assume that $2\leq m \leq 5$.

As a special case of Proposition 2.11 in \cite{D}, we have the following:
\begin{lemma}\label{prop2}
Let $\Dd$ be an arrangement of hyperplane sections in $Q$ and $\Hh_{\Dd}$ be its corresponding hyperplane arrangement in $\PP^{n+1}$. Then we have
\begin{equation}\label{eqa1}
0\to \Oo_Q(-2) \to \Omega_{\PP^{n+1}}^1 (\log \Hh_{\Dd} )_{\vert _Q}\to \Omega_Q^1 (\log \Dd) \to 0.
\end{equation}
\end{lemma}

 By Theorem \ref{DKthm} and the sequence (\ref{eqa1}), we have the following exact sequences with $\Ee:=\Omega_Q^1 (\log \Dd)$ if $\Dd$ is tame:
\begin{align*}
(m=2)&:0\to \Oo_Q(-2) \to \Oo_Q(-1)^{\oplus 2} \oplus \Oo_Q \to \Ee\to 0\\
(m=3)&: 0\to \Oo_Q(-2) \to \Oo_Q(-1) \oplus \Oo_Q^{\oplus 2} \to \Ee \to 0\\
(m=4)&: 0\to \Oo_Q(-2) \to \Oo_Q^{\oplus 3} \to \Ee \to 0\\
(m=5)&:0\to \Oo_Q(-2) \to T\PP^3(-1)_{\vert _Q} \to \Ee \to 0
\end{align*}

When $\Dd$ is not tame, we have the same exact sequences for $m=2,3,5$, while $\Omega_Q^1 (\log \Dd)$ is not globally generated if $m=4$ and so does not admit the sequence above (see Proposition \ref{ut1}).

It is known in \cite{DK} that the logarithmic bundle $\Omega_{\PP^n}^1 (\log \Hh)$ of hyperplanes $\Hh=\{H_1, \cdots, H_m\}$ with $m\geq n+2$, admits the Steiner resolution
$$0\to \Oo_{\PP^n}(-1)^{\oplus (m-n-1)} \to \Oo_{\PP^n}^{\oplus (m-1)} \to \Omega_{\PP^n}^1 (\log \Hh) \to 0, $$
whose restriction to $Q$ with the sequence (\ref{eqa1}) enables us to have
\begin{equation}
0\to \Oo_Q(-2)\oplus \Oo_Q(-1)^{\oplus (m-4)} \to \Oo_Q^{\oplus (m-1)} \to \Omega_Q^1(\log \Dd) \to 0
\end{equation}
with $n=3$.

Let $S=\CC[x_0, \cdots, x_3]$ be the coordinate ring of $\PP^3$  and let us consider the Euler sequence over $\PP^3$:
$$0\to \Oo_{\PP^3} \to \Oo_{\PP^3}(1)^{\oplus 4} \to T\PP^3 \to 0,$$
where the first map is defined by $1\mapsto \sum x_i \frac{\partial}{\partial x_i}$. In the exact sequence for the normal bundle
$$0\to TQ \to T\PP^3_{\vert_Q} \to \Nn_{Q|\PP^3} \to 0,$$
the second map $T\PP^3_{\vert_Q}\to \Nn_{Q|\PP^3}\cong \Oo_Q(2)$ is defined from the map $\Oo_{\PP^3}(1)^{\oplus 4}_{\vert_Q} \to \Oo_Q(2)$ sending $(a_0, \cdots, a_3)$ to $\sum a_i \frac{\partial F}{\partial x_i}$, where $F$ is the defining equation of $Q$.
Note that $\sum x_i \frac{\partial F}{\partial x_i}=2F(x_0, \cdots, x_3)$ and so it vanishes over $Q$.

\begin{proposition}\label{2conics}
Let $\Dd=\{D_1, D_2\}$ be an arrangement of two smooth conics on $Q$.
\begin{enumerate}
\item The map $\Phi_{(1,1)}^2$ is not generically injective.
\item The zeros of the unique section in $H^0(\Omega_Q^1 (\log \Dd))$ are the singular points of the two singular conics in the pencil spanned by $D_1$ and $D_2$.
\end{enumerate}
\end{proposition}
\begin{proof}
Let $\Hh=\Hh_{\Dd}=\{H_1, H_2\}$ be the corresponding hyperplane arrangement on $\PP^3$ and then we have
$\Omega_{\PP^3}^1(\log \Hh) \cong \Oo_{\PP^3}(-1)^{\oplus 2} \oplus \Oo_{\PP^3}$. Thus we have
\begin{equation}\label{seq3}
0 \to \Oo_Q(-2) \to \Oo_Q(-1)^{\oplus 2} \oplus \Oo_Q \to \Ee \to 0
\end{equation}
where $\Ee:=\Omega_Q^1 (\log \Dd)$. In particular we have $h^0(\Ee)=1$ and $h^0(\Ee(-i,-j))=0$ for $(i,j)\in \{(1,0), (0,1), (1,1)\}$. Thus we have
\begin{equation}\label{seq2}
0\to \Oo_Q \to \Ee \to \Ii_Z \to 0
\end{equation}
where $Z$ is a $0$-dimensional subscheme of $Q$ with length $2$. Since the dimension of $\Ext^1 (\Ii_Z, \Oo_Q)$ is $h^1(\Ii_Z(-2,-2))=2$, so there is a $5$-dimensional family of extension of type (\ref{seq2}). So the first assertion follows since the dimension of the family of two conics on $Q$ is $6$.

For the second assertion, let us assume that $x_i=0$ is the defining equation of $H_i$ for $i=1,2$. Then $T\PP^3( -\log \Hh_{\Dd})\subset T\PP^3$ consists of the vectors with the form $\{ a_0\frac{\partial}{\partial x_0}+ a_1x_1\frac{\partial}{\partial x_1}+ a_2x_2\frac{\partial}{\partial x_2}+ a_3\frac{\partial}{\partial x_3}\}$, i.e. it is the sheaf of holomorphic vector fields that are tangent to each $H_i$.
Since $\sum x_i \frac{\partial }{\partial x_i}=0$ in $T\PP^3$, so each element of $T\PP^3(-\log \Hh_{\Dd})$ can be expressed as
$$(a_0-a_2x_0)\frac{\partial}{\partial x_0}+ (a_1-a_2)x_1\frac{\partial}{\partial x_1}+ (a_3-a_2x_3)\frac{\partial}{\partial x_3}.$$
This gives us an isomorphism $T\PP^3(-\log \Hh_{\Dd}) \cong \Oo_{\PP^3}(1)^{\oplus 2}\oplus \Oo_{\PP^3}$. So the map $T\PP^3(-\log \Hh_{\Dd}) \to \Nn_{Q|\PP^3}$ is simply given by:
$$\Oo_Q(1)^{\oplus 2}\oplus \Oo_Q \to \Oo_Q(2)$$
$$(a,b,c) \to a \frac{\partial F}{\partial x_0}+b\frac{\partial F}{\partial x_3}
+cx_1\frac{\partial F}{\partial x_1}.$$
Dually the map $\pi : \Oo_Q(-2) \to \Oo_Q(-1)^{\oplus 2}\oplus \Oo_Q$ in (\ref{seq3}) is defined by sending $1$ to $(\frac{\partial F}{\partial x_0}, \frac{\partial F}{\partial x_3}, x_1\frac{\partial F}{\partial x_1})$. From the diagram below
$$\begin{array}{ccccccc}
& & &0& &0 & \\
& & &\downarrow & &\downarrow &\\
&          & &\Oo_Q   &  =  &\Oo_Q & \\
&& &\downarrow & &\downarrow & \\
0 \to &\Oo_Q(-2) &\stackrel{\pi}{\to}& \Oo_Q(-1)^{\oplus 2} \oplus \Oo_Q &\rightarrow &\Omega_Q^1 (\log \Dd)& \to 0\\
&\|& &\downarrow & &\downarrow &\\
0 \to &\Oo_Q(-2)& \to& \Oo_Q(-1)^{\oplus 2}&\rightarrow & \Ii_Z &  \to 0\\
&& &\downarrow & & \downarrow& \\

& &          &0& & 0,& \\
\end{array}$$
the $0$-dimensional subscheme $Z$ is the common zeros $\frac{\partial F}{\partial x_0}=\frac{\partial F}{\partial x_3}=0$ on $Q$. Note that the tangent plane at point $a=(a_0, \cdots, a_3)\in Q$ is given by $\sum \frac{\partial F}{\partial x_i}(a_0, \cdots, a_3)x_i=0$. For a conic in the pencil $\{ c_1x_1+c_2x_2\}$ to be a singular conic, i.e. its corresponding hyperplane is tangent to $Q$, we should have $\frac{\partial F}{\partial x_i}(a_0, \cdots, a_3)= 0$, $i=0,3$ for some $(a_0, \cdots, a_3)\in Q$ and so the second assertion follows.
\end{proof}


\section{multiprojective spaces}

Let $X=\PP^{n_1} \times \cdots \times \PP^{n_s}$ with $n_i>0$ be the product of $s$ projective spaces. Take multi-homogeneous coordinates $x_{10}, \cdots, x_{1n_1},x_{20},\ldots, x_{sn_s}$ and write $\Oo_X(a_1, \ldots, a_s)$ for the line bundle of multi-degree $(a_1, \ldots, a_s)$. Then we have
$$H^0(\Oo_X(a_1, \ldots, a_s))=\KK[x_{10},\ldots,  x_{sn_s}]_{a_1, \ldots ,a_s}$$
the vector space of all multi-homogeneous polynomials of multi-degree $\mathbf{a}=(a_1, \ldots, a_s)$. Simply denote it by $\KK[\mathbf{x}_1, \ldots, \mathbf{x}_s]_{\mathbf{a}}$ without confusion.

 Fixing $f\in \KK[\mathbf{x}_1, \ldots, \mathbf{x}_s]_{\mathbf{a}}$, let us define $J_i(f)\subseteq \KK[\mathbf{x}_1, \ldots, \mathbf{x}_s]_{\mathbf{a}-\mathbf{i}}$ to be the linear span of the polynomials $\frac{\partial}{\partial x_{ij}}(f)$, $0\le j \le n_i$. Here $\mathbf{i}=(0,\ldots, 0,1,0,\ldots, 0)$ is the unit vector with $1$ in $i$-th position.
We call $J_i(f)$ the {\it $i$-th Jacobian space} of $f$ or the type $\mathbf{a}-\mathbf{i}$ part of the {\it $i$-th Jacobian ideal}
of $f$.

Choose $f\in \KK[\mathbf{x}_1, \ldots, \mathbf{x}_s]_{\mathbf{a}}$ with $a_i>0$ and let us define a divisor $D= V(f)$. Assume that $D$ is smooth. We start with the exact sequence
\begin{equation}\label{equ1}
0 \to TX(-\log D) \to TX \to \Oo _D(\mathbf{a})\to 0.
\end{equation}
Let $E\subset X$ be any divisor of type $\mathbf{a}-\mathbf{i}$. Since $D$ is smooth and $E$ has type $\mathbf{a}-\mathbf{i}$, the scheme $E\cap D$ has codimension two in $X$. Thus
the equations of $D$ and $E$ form an exact sequence, so by tensoring with $\Oo _E$ the twist of (\ref{equ1}) by $\Oo_X(-\mathbf{i})$ we get the following exact sequence
\begin{equation}\label{equ2}
0 \to TX(-\log D)(-\mathbf{i})_{\vert_E} \to TX(-\mathbf{i})_{\vert_E} \to \Oo _{D\cap E}(\mathbf{a}-\mathbf{i})\to 0.
\end{equation}

\begin{proposition}\label{u1}
Assume that $h^1(TX(-\mathbf{a}))=0$. For two smooth divisors $D_1=V(f_1)$ and $D_2=V(f_2)$ with $f_i \in |\Oo_X(\mathbf{a})|$, $\mathbf{a}\gneq (1,\ldots, 1)$, we have
$$TX(-\log D_1) \not\cong TX(-\log D_2)$$
if  $J_i(f_1) \ne J_i(f_2)$ for some $1\leq i \leq s$.
\end{proposition}
\begin{proof}
Let $E\subset X$ be any divisor of type $\mathbf{a}-\mathbf{i}$. From the assumption and the standard exact sequence, the number $h^0(E, TX(-\mathbf{i})_{\vert_E})$ is constant over all such $E$. Let us consider a map from the sequence (\ref{equ2});
$$\phi_E : H^0(X, TX(-\mathbf{i})_{\vert_E}) \to H^0(E, \Oo_{D\cap E}(\mathbf{a}-\mathbf{i})).$$
We say that $E$ {\it jumps up} for $D$ or for $f$ if the rank of $\phi_E$ is smaller than the general rank. It is clear from the sequence (\ref{equ2}) that the equation of jumping $E$ for $D$ is in $J_i(f)$. So the assertion follows.
\end{proof}

\begin{remark}
By the Bott formula and the K\"unneth theorem, it is easily checked that the assumption $h^1(TX(-\mathbf{a}))=0$ is satisfied for all $\mathbf{a}\in \ZZ^{\oplus s}$ if $n_i\geq 3$ for all $i$. Similarly when $n_i\geq 2$ for all $i$, the assumption is satisfied if each $a_i$ is at least $3$.
\end{remark}

Now let us consider the case $s=2$ with each $n_i=1$, i.e. $X =Q= Q_2$.

The obvious pairs of $f$ of Sebastiani-Thom type \cite{uy} are the ones for
which there are coordinates $[x_0,x_1;y_0,y_1]$ such that $f = ux_0^ay_0^b + vx_1^ay_1^b$. The curve $D=V(f)$ is smooth if and only if
$uv\ne 0$. For smooth $D=V(f)$ and $D'=V(g)$ in $|\Oo_Q(a,b)|$, we have $TX(-\log D)\cong TX(-\log D')$ if $f= ux_0^ay_0^b + vx_1^ay_1^b$ and $g= u'x_0^ay_0^b + v'x_1^ay_1^b$ with non-zero $u,v,u',v'$. Indeed, the bundle $TQ(-\log D)$ is the sheafification of
$$D_0(- \log f)=\{\delta \in \mathrm{Der}_A ~|~ \delta f=0\}$$
where $A=K[x_0,x_1;y_0,y_1]$. If we denote $\delta$ by $\sum_i u_{xi}\partial / \partial x_i + u_{yi}\partial / \partial y_i$, then we have
$$\delta f=u(u_{x0}\frac{\partial}{\partial x_0}+u_{y_0}\frac{\partial}{\partial y_0})(x_0^ay_0^b)+v(u_{x1}\frac{\partial}{\partial x_1}+u_{y_1}\frac{\partial}{\partial y_1})(x_1^ay_1^b)$$
and so $\delta f=0$ is independent of choices of $(u,v)\in (K^{\times})^2$. In particular the map $D \mapsto TX(-\log D)$ is not Torelli.
Then we may ask if, up to linear change of coordinates, this is the only way
to have the same logarithmic bundles.

\begin{theorem}\label{mq}
Let $D=V(f)$ and $D'=V(g)$ be smooth elements in $|\Oo_Q(a,b)|$ with $(a,b)\gneq (2,2)$. Then we have
$\Omega _Q^1 (\log D) \cong \Omega _Q^1(\log D')$ if and only if $f = x_0^ay_0^b +x_1^ay_1^b$ and $g =  x_0^ay_0^b +\lambda x_1^ay_1^b$
with $\lambda \ne 0$ for some coordinates $[x_0, x_1;y_0,y_1]$ of $Q$.
\end{theorem}

\begin{proof}
The ``~if ~'' part is obtained from the argument above.

Conversely let us start with some system of coordinates and then by Proposition \ref{u1} we have $J_1(f) =J_1(g)$ and $J_2(f) =J_2(g)$. In the pencil of $|\Oo _Q(a,b)|$ spanned by $D$ and $D'$
there is at least one singular curve, say $D''=V(w)$, because $\Oo _Q(a,b)$ is very ample. Writing $w=uf+vg$ with $u,v$ scalars, we have $u\ne 0$ and $v\ne 0$ since $D$ and $D'$ are smooth. Note that $J_1(w) \subseteq J_1(f)$ and $J_2(w)\subseteq J_2(f)$.

Let $P$ be a singular point of $D''$ and then all the partial derivatives of $w$ vanishes at $P$. Assume $J_i(w) =J_i(f)$ for all $i\in \{1,2\}$. It implies that all the partial derivatives of $f$ also vanish at $P$. Since $f$ is bihomogeneous, the Euler relation gives $P\in D$. Thus $P$ is a singular point of $D$, a contradiction. Hence there is an integer $i\in \{1,2\}$ such that $J_i(w) \ne J_i(f)$. Take for instance $i=1$ and it means
that $\partial w/\partial x_0$ and $\partial w/\partial x_1$ are not linearly independent. Up to a linear change of the coordinates $[x_0, x_1]$
we may assume $\partial w/\partial x_1 =0$ and so we have $w = x_0^ar(y_0,y_1)$ for some nonzero $r\in K[y_0,y_1]_b$. Since $J_2(w)$ is spanned by polynomials $x_0^a\partial r/\partial y_i$, $i=0,1$, not containing $x_1$, we have $J_2(w)\subsetneq J_2(f)$. As above we see
that, up to a linear change of the coordinates $[y_0, y_1]$ we may take $w= x_0^ay_0^b$.

Set $h:= f-w$. Since $J_1(w) \subset J_1(f)$, there are $e, e' \in K$
such that $ax_0^{a-1}y_0^b = e\partial f/\partial x_0 + e'\partial f/\partial x_1$. Since $D$ is smooth, $f$ is not divisible by $x_0$. Since $a\ge 2$, we get $e\ne 0$ and so there is a linear change of coordinates
$$x_0 \mapsto x_0 ~,~x_1 \mapsto e_1x_0+e'_1x_1$$
which does not change the formula for $w$. In these coordinates we have $\partial f/\partial x_0 = \partial w/\partial x_0$,
i.e. $\partial h/\partial x_0 =0$. Thus we may assume that $h=x_1^ad(y_0, y_1)$, up to multiplication by $x_1$. Since $u\ne 0$ and $v\ne 0$, we get the existence of $\lambda \in K\setminus \{0\}$ such that $f= x_0^ay_0^b + h$ and $g = x_0^ay_0^b + \lambda h$, up to the multiplication
of $g$ by a non-zero scalar. Since $J_2(w) \subset J_2(f)$, there are constants $s, s'$ such that $\partial w/\partial y_0 = s\partial f/\partial y_0 +s'\partial f/\partial y_1$.
Since $D$ is smooth, $f$ is not divisible by $y_0$. Since $b\ge 2$, we get $s \ne 0$ and so there is a linear change of coordinates
$$y_0\mapsto y_0~,~y_1\mapsto s_1y_0+s_1'y_1$$
which does not change the formula for $w$. Thus we have $\partial w/\partial y_0 = \partial f/\partial y_0$,
i.e. $\partial h/\partial y_0=0$ and so we have $d(y_0,y_1) = cy_1^b$ for some $c\in K$. Since $D$ is smooth, we have $c\ne 0$. Taking $c_1y_1$ instead of $y_1$ for $c_1\in K$ such that $c_1^b =c$, we get $f = x_0^ay_0^b +x_1^ay_1^b$ and $g =  x_0^ay_0^b +\lambda x_1^ay_1^b$
\end{proof}

\begin{remark}
We call equations of the form $ux_0^ay_0^b+vx_1^ay_1^b$ for some $u,v\in K$ {\it equations of split type}.
\end{remark}

\begin{theorem}\label{u2}
For $X=\PP^{n_1}\times \cdots \times \PP^{n_s}$ with $s\geq 2$, the map
$$D \mapsto TX(-\log D)$$
is generically injective for $D\in |\Oo_X(\mathbf{a})|$ with $\mathbf{a}\geq (3,\ldots, 3)$.
\end{theorem}

\begin{proof}
First let us assume that there exists $i$ such that $n_i \geq 2$, say $n_1$.
For a fixed generic divisor $D = V(f)$, let us assume the existence of a smooth $D' = V(g)$ with $TX (-\log D) \cong TX (-\log D')$ and then we have $J_i(f) =J_i(g)$ for each $i$ by Proposition \ref{u1}.

For a general point $P \in \PP^{n_2} \times \cdots \times \PP^{n_s}$, let us define $X_P:= \PP^{n_1}\times \{P\}$. Since our base field has characteristic zero, the restrictions of the projection $X\to \PP^{n_1}$ to $D$ and $D'$ give us that $D_P:= D\cap X_P$ and $D'_P:= D'\cap X_P$ are smooth due to the general smoothness.
Note that they have homogeneous equations $f_P(x_{10},\dots ,x_{1n_1}):= f(x_{10},\dots ,x_{1n_1},P)$ and $g_P(x_{10},\dots ,x_{1n_1}):=g(x_{10},\dots ,x_{1n_1},P)$ and that $J(f_P)= J(g_P)$ since $J_1(f)=J_1(g)$. From the generality of $f$ and $g$ together with the point $P$, we obtain that $f_P$ and $g_P$ are general polynomials of degree $d$ at least $3$. Note that a general polynomial is not of Sebastiani-Thom type. Indeed, by separating $n+1$ variables $x_0, \cdots, x_n$ divided into two parts, say $x_0, \ldots, x_l$ and $x_{l+1}, \ldots, x_n$, the dimension of Sebastiani-Thom type polynomials of degree $d\geq 3$ can be computed to be
$$\dim \mathrm{PGL}_{n+1}+ {l+d \choose d} + {n-l-1+d \choose d}.$$
This number is less than ${n+d\choose d}$, the dimension of homogeneous polynomials of degree $d$. Thus we have $V(f_P)=V(g_P)$ and so we have $D=D'$ since $P$ is general.

Now assume that each $n_i=1$. The case $s=2$ is derived from Theorem \ref{mq}.
Assuming $s\geq 3$, let us denote $Y=\PP^{n_3}\times \cdots \times \PP^{n_s}$. For a divisor $D=V(f)$ of $X$ and a general point $b\in Y$, the divisor $D_b$ of $\PP^1 \times \PP^1$ is defined to be the intersection of $D$ with $\PP^1 \times \PP^1 \times \{b\}$, i.e. we have $D_b=V(f_b)$ where $f_b$ is the bihomogeneous equation, evaluated at $b$. If we have $TX(-\log D)\cong TX(-\log D')$ for another general divisor $D'$ of $X$ with $D'=V(f')$, we have $J_i(f_b)=J_i(f'_b)$. Again by Theorem \ref{mq} it contradicts to the generality of $D,D'$.
\end{proof}

\providecommand{\bysame}{\leavevmode\hbox to3em{\hrulefill}\thinspace}
\providecommand{\MR}{\relax\ifhmode\unskip\space\fi MR }
\providecommand{\MRhref}[2]{%
  \href{http://www.ams.org/mathscinet-getitem?mr=#1}{#2}
}
\providecommand{\href}[2]{#2}


\begin{thebibliography}{10}

\bibitem{Angelini}
E.~Angelini, \emph{Logarithmic bundles of hypersurface arrangements in $\PP^n$}, Preprint, arXiv : 1304.5709, 2013.

\bibitem{De}
P.~Deligne, \emph{Th\'eorie de Hodge II}, Inst. Hautes \'Etudes Sci. Publ. Math. \textbf{40} (1971), 5--58.


\bibitem{D}
I.~Dolgachev, \emph{Logarithmic sheaves attached to arrangements of
  hyperplanes}, J. Math. Kyoto Univ. \textbf{47} (2007), no.~1, 35--64.


\bibitem{DK}
I.~Dolgachev and M.~Kapranov, \emph{Arrangements of hyperplanes and vector
  bundles on {$\bold P\sp n$}}, Duke Math. J. \textbf{71} (1993), no.~3,
  633--664.

  \bibitem{eisenbud} D.~Eisenbud, \emph{Commutative algebra: With a view toward algebraic geometry}, Graduate Texts in Mathematics \textbf{150}, Springer, New York, 1995.
      
      \bibitem{Kn}  H. Kn\"{o}rrer,
\emph{Cohen-Macaulay modules of hypersurface singularities I}, Invent. Math. 88 (1987), 153-164.


\bibitem{FMV}
D.~Faenzi, D.~Matei and J.Vall\`es, \emph{Hyperplane arrangements of Torelli type}, Compositio Math. \textbf{149} (2013), no.~2, 309--332.


\bibitem{Huh}
S.~Huh, \emph{Moduli of stable sheaves on a smooth quadric and a Brill-Noether locus}, J. Pure Appl. Algebra \textbf{215} (2011), no. ~9, 2099--2105.

\bibitem{uy}
K.~Ueda and M.~Yoshinaga, \emph{Logarithmic vector fields along smooth
divisors in projective spaces} Hokkaido Math. J. \textbf{38} (2009), no.~3, 409--415.

\bibitem {Valles}
J.~Vall\`es, \emph{Nombre maximal d'hyperplans instables pour un fibr\'e de Steiner}, Math. Z. \textbf{233} (2000), 507--514.

\end{thebibliography}
\end{document}